\documentclass[12pt]{amsart}
\usepackage{fullpage,prefix1,wrapfig,xcolor,tikz,caption}
\usetikzlibrary{shapes.geometric,patterns}

\title{Causally Disjoint Discs: Another $\mathbb{E}_n$-operad}

\author{Ryan Grady}

\address{Department of Mathematical Sciences\\Montana State University\\Bozeman, MT 59717}
\email{ryan.grady1@montana.edu}



\begin{document}

\begin{abstract}
Motivated by (perturbative) quantum observables in Lorentzian signature we define a new operad: the operad of causally disjoint disks.  In order to describe this operad we use the orthogonal categories of Benini, Schenkel, and Woike and the prefactorization functor of Benini, Carmona, Grant-Stuart, and Schenkel.   Along the way we extend these constructions to the topological setting, i.e., (multi-)categories enriched over spaces.
\end{abstract}

\maketitle

\tableofcontents

\section{Introduction}

In the Euclidean setting there is a beautiful equivalence between locally constant factorization algebras and $\mathbb{E}_{\mathsf{n}}$-algebras, where $n$ is the dimension of space.  Such algebras arise as algebras of observables in topological (quantum) field theory, see \cite{CG} and \cite{L}.

Inspired by the work of Gwilliam and Rejzner, \cite{GR1} and \cite{GR}, we are pursuing a similar relationship in the Lorentzian setting. More specifically, the setting is that of perturbative algebraic quantum field theory.  There has been much work by Benini, Schenkel and collaborators on the operadic and homotopical structure of observables in this setting.  In this setting, analytic (and other) issues are subtle and require significant care, even in the free case, see for instance \cite{BS} and \cite{BMS}. We take a slightly different approach and consider the underlying (differential) topological structure.

To this end, we introduce the operad of \emph{causally disjoint discs} which is a version of the operad of little $n$-discs which incorporates the causal structure induced by a Lorentzian structure.  This operad which we denote by $\cP_{\mathsf{CD_n}}$ receives a map from little $(n-1)$-discs and maps to little $n$-discs.  The former map is actually a weak equivalence (Theorem \ref{thm}).  At the level of algebras, the composition of these maps is equivalent to that induced by the natural map of little $(n-1)$-discs into little $n$-discs. We also (briefly) discuss a version of our results for causally disjoint causal diamonds.

In constructing the operad of causally disjoint discs, we utilize the orthogonal categories and prefactorization functor of Benini--Schenkel and company which appear in \cite{BS} and \cite{BCS} respectively.  More accurately, we extend the definition of orthogonal category, functor, etc. to the setting of topological categories, i.e., those enriched over the category of spaces $\mathsf{Top}$.  We then define the prefactorization operation which associates a colored operad to an orthogonal category and prove functoriality of this operation.

We finish by describing a couple examples of causally disjoint disc algebras and their relation to $\mathbb{E}_\mathsf{n}$-algebras, before offering some speculation about using causally disjoint discs to describe Wick Rotation for ``sufficiently topological" pAQFTs.

\subsection*{Acknowledgements} I thank Victor Carmona for introducing me to orthogonal categories and David Ayala for several useful conversations.  Moreover, I am grateful to Alex Schenkel for some detailed comments on an earlier version of this note; especially the suggestion to consider causally convex neighborhoods as in Section \ref{sect:cd}. Finally, I thank the anonymous referee for their careful reading and for the suggestion to consider the connection to the locality relations of Guo, Paycha, and Zhang, see Remark \ref{rem:locality}. 

\subsection*{Funding Declaration} The author is supported by the Simons Foundation under Travel Support/Collaboration 9966728.

\subsection*{Data Availability} No data was generated by this work.

\subsection*{Competing Interests} The author states no competing interests.

\subsection*{Author Contribution} R.G.~is the sole author.

\subsection*{Ethics Declaration} Not applicable. 

\section{Preliminaries and conventions}

\subsection{(Colored)-Operads}

Throughout a colored operad will be defined as a multicategory.  We follow the conventions of \cite{EM} for multicategories.  In particular, our multicategories need not be small, but rather can have a (proper) class of objects.  Additionally, all multicategories will be symmetric.

An operad is simply a colored operad with one color/object.  An object of central interest is the operad of little $n$-discs, which is an $\mathbb{E}_n$ operad.  Much about this operad, e.g., its homology, is covered in Chapter 4 of \cite{F}. Moreover, $\mathbb{E}_n$ and several operads appearing in this note are \emph{topological operads}, i.e., they arise via multicategories enriched over topological spaces.   We will often ask our $\mathsf{Top}$-categories to be further tensored/copowered over $\mathsf{Top}$; the article \cite{K} contains all relevant details.  (We expect that our developments will work over a suitably nice symmetric monoidal category $(\cM, \otimes)$ not just $(\mathsf{Top}, \times)$.)

Finally, all spaces of (smooth) maps between manifolds will be equipped with the (smooth) compact open topology.  This topology agrees with the weak Whitney topology provided the domain is compact.

\subsection{Causal structure}

Throughout we will work on $n$-dimensional Minkowski space, $\RR^{1,n-1}$, i.e., the manifold $\RR^n$ equipped with its standard Lorentzian metric $\eta$ of signature $(1,n-1)$.  The space $\RR^{1,n-1}$ is time oriented and we will utilize the resulting causal structure.

\begin{definition}
A differentiable curve $\gamma \colon I \to \RR^{1,n-1}$ is \emph{causal} if at each time $t \in I$, $\gamma'(t)$ is timelike or null, i.e., 
\[
\eta(\gamma'(t), \gamma'(t)) \le 0, \quad \forall t \in I.
\]
\end{definition}

\begin{definition}
Let $S,T \subseteq \RR^{1,n-1}$ be subsets.  The subsets $S$ and $T$ are \emph{causally disjoint} (or \emph{causally separated}) if there exists no causal curves between any point of $S$ and any point of $T$.
\end{definition}

  \begin{lemma}\label{lem:cdprop}
  Causal disjointness defines a symmetric and transitive relation on the subsets of $\RR^{1,n-1}$.
  \end{lemma} 

  An equivalent way to think of causally disjoint subsets is that $S$ is not contained in the union of future light cones over the points of $T$, nor is $T$ in the union of future light cones over $S$. Or, $S$ is not contained in the union of future and past light cones of points in $T$; symmetrically, $T$ is not contained in the union of forward and backward light cones of $S$. See Figures \ref{fig:1} and \ref{fig:2} below.
  
  \vspace{1ex}
  
  \begin{center}
\begin{minipage}{.5\textwidth}
\centering
       \begin{tikzpicture}[scale=0.6]
              \fill[pattern=north west lines, pattern color = gray!75]plot coordinates{(-10.1,0)(-12.1,2)(-3.9,2)(-5.9,0)(-10.1,0)}; 
                            \fill[pattern=north west lines, pattern color = gray!75] plot coordinates{(-10.1,0)(-12.1,-2)(-3.9,-2)(-5.9,0)(-10.1,0)}; 
       \draw[gray!60,fill=gray!60] (-8,0) ellipse (2cm and 0.5cm);
       \draw node at (-8,0) {$S$};
       \draw[gray!60,fill=gray!20] (-3.5,.5) circle (1cm);
       \node at (-3.5,.5) {$T$};
       \end{tikzpicture}
\captionof{figure}{$S$ and $T$ are causally disjoint subsets.}\label{fig:1}
    \end{minipage}%
    \begin{minipage}{0.5\textwidth}
        \centering
  \begin{tikzpicture}[scale=0.6]
   \draw[gray!60,fill=gray!20] (-4.5,1) circle (1cm);
       \node at (-4.5,1) {$T$};
              \fill[pattern=north west lines, pattern color = gray!75]plot coordinates{(-10.1,0)(-12.1,2)(-3.9,2)(-5.9,0)(-10.1,0)}; 
                            \fill[pattern=north west lines, pattern color = gray!75] plot coordinates{(-10.1,0)(-12.1,-2)(-3.9,-2)(-5.9,0)(-10.1,0)}; 
       \draw[gray!60,fill=gray!60] (-8,0) ellipse (2cm and 0.5cm);
       \draw node at (-8,0) {$S$};
      
       \end{tikzpicture}
\captionof{figure}{The subsets $S$ and $T$ are not causally disjoint.}\label{fig:2}    \end{minipage}
\end{center}

\section{Orthogonal categories and prefactorization}

In this section we will recall and then extend definitions of orthogonal categories and the prefactorization functor as in \cite{BS} and \cite{BCS} respectively.

\subsection{Orthogonal categories}

Orthogonal categories were developed in \cite{BS} in order to codify commutation relations in algebraic quantum field theory.  We will see that they also provide a convenient way to describe/construct (colored)-operads.

\begin{definition}
An \emph{orthogonal category} is a pair $(C,\perp)$ where $C$ is a locally small category and $\perp \subseteq \mathop{\mathrm{Mor}} C \times_{\mathop{\mathrm{Obj}} C} \mathop{\mathrm{Mor}} C$ (common target) such that $\perp$ is symmetric and stable under composition, i.e., for all composable morphisms $g, h_1, h_2$ and $(f_1, f_2) \in \perp$, then
$(g \circ f_1 \circ h_1 , g \circ f_2 \circ h_2) \in \perp$.
\end{definition}

From now on we will also assume that our categories $C$ have finite coproducts and products.

\begin{example}
The original example of an orthogonal category in \cite{BS} is $(\mathbf{Loc}, \perp)$ where $\mathbf{Loc}$ is the category of Lorentzian manifolds and orientation and time-orientation preserving isometric embeddings, with a pair of morphisms with common target $(f_1, f_2)$ in $\perp$ if and only if their images are causally disjoint.  There is also a subcategory $\mathbf{Loc^{rc}} \subset \mathbf{Loc}$ which consists of morphisms that are further Cauchy or relatively compact; this subcategory plays a central role in \cite{BCS}.
\end{example}

\begin{remark}\label{rem:locality}
The work of Guo, Paycha, and Zhang \cite{GPZ} is also inspired by locality in quantum field theory and generalizes causal disjointness.  In their work, a \emph{locality relation} is simply a symmetric bilinear relation; more significantly, the authors develop the notion of \emph{locality monoid} and \emph{locality algebra} (complete with a theory of ideals).  The relationship between (small) orthogonal categories and locality sets/relations is as follows:
\begin{itemize}
\item[(a)] If $(C,\perp)$ is a small orthogonal category, then the set of morphisms $\mathop{\mathrm{Mor}} C$ inherits a locality relation from the orthogonality relation $\perp$.
\item[(b)] Composition of morphisms (in $C$) does not, in general, define a locality semi-group structure.
\item[(c)] However, for a fixed object $c \in C$, composition does define a locality semi-group structure on $\mathrm{Hom}_C (c,c)$, and provided $\perp \neq \emptyset$, then the automorphisms of the given object, $\mathrm{Aut}_C (c,c)$, is a locality monoid (group even). 

\end{itemize}

\end{remark}

\begin{definition}\label{defn:oFunctor} A \emph{morphism between orthogonal categories} $(B, \perp_B)$ and $(C, \perp_C)$ is a functor $F \colon B \to C$ which preserves finite products and coproducts and is $\perp$-compatible, i.e., $F(\perp_B) \subseteq \perp_C$.
\end{definition}

Hence, we have a category of orthogonal categories: $\mathsf{\perp \hspace{-0.75ex}Cat}$.  It is straightforward to enrich this notion over the category of topological spaces $\mathsf{Top}$.  We will denote the resulting category of enriched orthogonal categories by $\mathsf{Top}\text{--}\mathsf{\perp \hspace{-0.75ex}Cat}$.

\subsection{The functor of prefactorization}

Building on \cite{BCS}, we define a functor which takes orthogonal categories to colored operads (multicategories). 

\begin{definition}
Let $(C, \perp)$ be an orthogonal category enriched in $\mathsf{Top}$ (and which has finite coproducts).  The \emph{prefactorization operad associated to} $(C,\perp)$ is the multicategory, enriched in $\mathsf{Top}$, $\cP_C^\perp$ with the same objects as $C$ and morphisms
\[
\cP_C^\perp \begin{pmatrix} c_1, c_2, \dotsc, c_k \\ t \end{pmatrix} = \left \{ (f_i) \in \prod_{i=1}^k \mathrm{Hom}_C (c_i , t) : (f_i , f_j) \in \perp \text{ for } i \neq j \right \},
\]
with composition inherited from that of $C$ and where permutations act naturally on the  products of morphisms.
\end{definition}

\begin{prop}
If $(C, \perp)$ is an orthogonal category enriched in $\mathsf{Top}$, then $\cP_C^\perp$ is a (symmetric) multicategory enriched in $\mathsf{Top}$.
\end{prop}

\begin{proof}
That composition is well-defined follows from requirement that $\perp$ is stable under composition.  Indeed, for a multimorphism $\vec{f} = (f_i)_{i=1}^k$ and a $k$-tuple of (composable) multimorphisms $(\vec{g_1}, \vec{g_2}, \dotsc , \vec{g_k})$, composition is given by
\[ 
\vec{f} (\vec{g_1} , \dotsc , \vec{g_n} ) = (f_1 \circ g_{11}, f_1 \circ g_{12}, \dotsc, f_k \circ g_{k j_k} ),
\]
so $f_i \circ g_{i \ell} \perp f_i \circ g_{i \ell'}$ as $g_{i \ell} \perp g_{i \ell'}$ and $f_i \circ g_{i \ell} \perp f_{i'} \circ g_{i' \ell'}$ as $f_i \perp f_{i'}$.
Since transpositions generate the symmetric groups and $\perp$ is symmetric, the action of $\Sigma_k$ on the collection of $k$-ary maps is well-defined.  Finally, that composition is equivariant with respect to the symmetric group actions follows from the enriching category $\mathsf{Top}$ being symmetric monoidal.
\end{proof}

\begin{remark}
A word of caution, in \cite{BS} there is another colored operad associated to an orthogonal category $(C, \perp)$ which is denoted $\sO_C$, $\sO_C$ and $\cP^\perp_C$ are not the same!  Indeed, it follows from Theorem 2.9 of \cite{BPSW} that $\sO_C \simeq \cP^\perp_C \otimes_{\mathrm{BV}} \mathsf{As}$, where $\mathsf{As}$ is the (unital) associative operad and $\otimes_{\mathrm{BV}}$ is the Boardman--Vogt tensor product.
\end{remark}

\begin{remark}
In \cite{BCS}, the authors construct another prefactorization operad which incorporates \emph{time-orderability}.  This operad does not come from an orthogonal category as time-orderability is not a binary relation.
\end{remark}

\begin{prop}
Prefactorization actually defines a functor $\cP^\bullet_\bullet \colon  \mathsf{Top}\text{--}\mathsf{\perp \hspace{-0.75ex}Cat} \to \mathsf{Top}\text{--}\mathsf{MultiCat}$.
\end{prop}

\begin{proof}
Let $F \colon (B, \perp_B) \to (C, \perp_C)$ be a functor of (topological) orthogonal categories. That $\cP(F) \colon \cP_B^{\perp_B} \to \cP_C^{\perp_C}$ is well defined at the level of objects and sets of multimorphisms follows directly from the definition of functor between orthogonal categories (Definition \ref{defn:oFunctor}).  Moreover,  composition is inherited from the orthogonal categories themselves, so since $F$ intertwines composition, so does $\cP(F)$.  Finally, symmetric group equivariance is a (slightly) tedious verification which boils down to $F$ being a functor of enriched categories and compatible with the product of mapping spaces.
\end{proof}

\begin{cor}
If $F \colon (B, \perp_B) \to (C, \perp_C)$ is a functor of orthogonal categories and $\cD$ is any symmetric monoidal category (enriched and tensored over $\mathsf{Top}$), then $F$ induces a map at the level of algebras
\[
F^\ast \colon \mathsf{Alg}_{\cP_C^{\perp_C}} (\cD) \to \mathsf{Alg}_{\cP_B^{\perp_B}} (\cD).
\]
\end{cor}

\section{Causally disjoint discs}

Let $D^n$ be the Euclidean unit disc in $\RR^{1,n-1}$, where the latter is equipped with its standard Lorentzian structure.  Recall that a \emph{rectilinear} embedding $\varepsilon \colon D^n \to D^n$ is a smooth (topological) embedding of the form $\varepsilon(\vec{x})=r\vec{x} + \vec{b}$ for $r \in \RR_{>0}$ and $\vec{b} \in \RR^n$. 

\begin{definition}
Let $n \in \NN$.  Define \emph{the orthogonal category of causally disjoint discs} as follows.  The category $\mathsf{CD_n}$ has one object, $\mathrm{Hom} (\ast, \ast) = \mathrm{Rect}(D^n, D^n)$, and $(f,g) \in \perp_{\mathsf{CD_n}}$ if and only if $f(D^n)$ and $g(D^n)$ are causally disjoint (in $\RR^{1, n-1}$).
\end{definition}

Recall our convention that $\mathrm{Rect}(D^n, D^n)$ is equipped with the compact open topology, so $\mathsf{CD_n}$ is an orthogonal category enriched over $\mathsf{Top}$.

\begin{lemma}
The orthogonal category $\mathsf{CD_n}$ is well-defined.
\end{lemma}

This lemma follows from a simple observation in Lorentzian geometry.

\begin{lemma}
Let $D \subset \RR^{1,n-1}$ be the unit disc with respect to the Euclidean metric on $\RR^n$.  Let $S,T \subseteq D$ be subsets which are causally disjoint with respect to the ambient causal structure.  For any rectilinear embedding $\varphi \colon D \hookrightarrow \RR^{1,n-1}$ the images $\varphi (S)$ and $\varphi (T)$ are causally disjoint subsets.
\end{lemma}

\begin{proof}
By definition, a rectilinear map is of the form $\vec{x} \mapsto r \vec{x} +\vec{b}$. Translations are contained in the Poincar\'e group and scaling is conformal, so such a map preserves the causal structure. 
\end{proof}

Similarly, one defines an orthogonal category $\mathsf{Disc_n}$ exactly as $\mathsf{CD_n}$ except that $(f,g) \in \perp_{\mathsf{Disk_n}}$ if and only if $\mathop{\mathrm{int}} f(D^n)$ and $\mathop{\mathrm{int}} g(D^n)$ are (set theoretically) disjoint.  Note that there is an obvious forgetful functor $\omega \colon \mathsf{CD_n} \to \mathsf{Disc_{n-1}}$.  By functoriality, there is a map of colored operads for which we will abuse notation and also denote by $\omega \colon \cP_{\mathsf{CD_n}} \to \cP_{\mathsf{Disc_n}}$.  Finally, it is immediate that $\cP_{\mathsf{Disc_n}}$ is an $\mathbb{E}_\mathsf{n}$-operad, e.g., by comparing directly to the definition in \cite{F}.

\subsection{An equivalence between $\cP_{\mathsf{CD_n}}$ and $\cP_{\mathsf{Disc_{n-1}}}$}


For $n, k \in \NN, n \ge 2$, consider the embedding 
\[
\varepsilon_{n,k} \colon \mathrm{Rect} ( \underbrace{D^{n-1} \amalg \dotsb \amalg D^{n-1}}_{k \text{ cofactors}}, D^{n-1}) \to \mathrm{Rect} ( \underbrace{D^{n} \amalg \dotsb \amalg D^{n}}_{k \text{ cofactors}}, D^n ),
\] 
where $\iota_{n,k}$ is defined as follows. For $\varphi \in \mathrm{Rect} ( D^{n-1} \amalg \dotsb \amalg D^{n-1}, D^{n-1}) $, let $\varphi_j \colon D^{n-1} \to D^{n-1}$ be the restriction to the $j$th cofactor.  Further, let $\iota \colon D^{n-1} \hookrightarrow D^{n} \subseteq \RR^{1,n-1}$ be the canonical inclusion in the $\{t=0\}$ hyperplane, and $s \colon D^{n-1} \hookrightarrow D^n$ the precomposition of $\iota$ with scalar multiplication by $\frac{1}{2}$. Next, let $\Sigma \left ( (s \circ \varphi_j)(D^{n-1}) \right )$ be the double cone (in the $t$-direction) with cone angle $\frac{\pi}{4}$, and $D^n_{\varphi_j}$ the maximal inscribed $n$-disc in $\Sigma \left ( (s \circ \varphi_j)(D^{n-1}) \right )$.  Finally, define the restriction of $\varepsilon_{n,k} (\varphi)$ to the $j$th cofactor to be the unique rectilinear map sending $D^n$ to $D^n_{\varphi_j}$.

The embedding $\varepsilon_{n,k}$ is a small twist on the standard embedding of little $(n-1)$-discs into $n$-discs to guarantee causal disjointness.

\begin{prop}
For all $n,k \in \NN$, the map $\varepsilon_{n,k}$ is a well-defined (topological) embedding.  Further, if $\varphi  \in \mathrm{Rect} ( D^{n-1} \amalg \dotsb \amalg D^{n-1}, D^{n-1})$ is such that the cofactors have disjoint images, then the images of cofactors under $\varepsilon_{n,k}(\varphi)$ are causally disjoint.
\end{prop}

\begin{proof}
The first claim follows from the universal property of coproducts and the definition of rectilinear embeddings.  The second claim is an elementary exercise in the geometry of $n$-dimensional Minkowski space.
\end{proof}

\begin{theorem}\label{thm}
For $n \ge 2$, the operads $\cP_{\mathsf{CD_n}}$ and $\cP_{\mathsf{Disc_{n-1}}}$ are homotopy equivalent.
\end{theorem}


The theorem follows from a simple check that the embeddings $\varepsilon_{n,k}$ and the retractions $R_{n,k}$ defined below assemble to maps of (symmetric) operads.  To begin, we note the following which follows from a straightforward argument in (Euclidean) geometry.

\begin{lemma}
Let $D_1:= D^n_{r_1} (c_1, c_2, \dotsc ,c_n)$ and $D_0:=D^n_{r_2}(0,0, \dotsc, 0)$ be Euclidean discs of radius $r_1$ and $r_2$ centered at $\vec{c}, \vec{0} \in \RR^{1,n-1}$, with $c_1 \ge 0$. Let $\delta \in [0, c_1)$. If $D_1$ and $D_0$ are causally disjoint, then the $\delta$-shift in the $t$-direction of $D_1$ and $D_0$ are causally disjoint, where the $\delta$-shift is given by $D_{1,\delta}:=D^n_{r_1} (c_1-\delta, c_2, \dotsc ,c_n)$.
\end{lemma}

We stated the lemma with particular hypotheses about the centers of the discs for convenience, but as (ambient) translations preserve causal structure, the lemma is just a statement about the relative positioning of the centers of discs and decreasing the $t$-distance between these centers preserving causal disjointness.

\begin{lemma}
For $n, k \in \NN, n \ge 2$, the space $\cP_{\mathsf{CD_n}} (k)$ deformation retracts onto the space $\varepsilon_{n,k} \left ( \cP_{\mathsf{Disc_{n-1}}} (k) \right )$.
\end{lemma}

\begin{proof}
For given $n$ and $k$ we will define a map
\[
R_{n,k} \colon \cP_{\mathsf{CD_n}} (k) \times [0,1]  \to \cP_{\mathsf{CD_n}} (k).
\]
To begin, let 
\[
\widetilde{R_{n,k}} \colon \cP_{\mathsf{CD_n}} (k) \times [0,1]  \to \cP_{\mathsf{CD_n}} (k)
\]
be the ``slide to the $\{t=0\}$ plane map." That is, for $\psi \in \cP_{\mathsf{CD_n}} (k) \subset  \mathrm{Rect} ( \underbrace{D^{n} \amalg \dotsb \amalg D^{n}}_{k \text{ cofactors}}, D^n )$, $\widetilde{R_{n,k}} (\psi)$ is the path which moves the image of each disc to the $\{t=0\}$ plane along the line orthogonal to the plane from the center of the disc with speed parametrized by arc length. That this map preserves causal disjointness follows from repeated application of the preceding lemma.

The map $\widetilde{R_{n,k}}$ is nearly a (strong) deformation retract onto $\varepsilon_{n,k} \left ( \cP_{\mathsf{Disc_{n-1}}} (k) \right )$, as it fixes this subspace pointwise.  The remaining issue is that we shrank the $(n-1)$-discs in defining the embedding $\varepsilon_{n,k}$ in order to ensure causal disjointness (the shrinking factor was convenient, though far from optimal), so it is not the case that $\widetilde{R_{n,k}} (-,1) \left ( \cP_{\mathsf{CD_n}} (k) \right ) \subseteq \varepsilon_{n,k} \left ( \cP_{\mathsf{Disc_{n-1}}} (k) \right )$.  Nonetheless, these two images are homotopy equivalent.  Indeed, indexing the discs from left to right there is a homotopy equivalence (in the space of causally disjoint rectilinear embeddings) which takes the image of the left most disc under $\widetilde{R_{n,k}}$ to the image of the left most disc under $\varepsilon_{n,k}$ while fixing the others; the homotopy is just ``slide and rescale."  Proceeding from left to right and concatenating these maps defines a (strong) deformation rectraction, $H$, from $\widetilde{R_{n,k}} (-,1) \left ( \cP_{\mathsf{CD_n}} (k) \right )$ onto $\varepsilon_{n,k} \left ( \cP_{\mathsf{Disc_{n-1}}} (k) \right )$.  Hence, our desired map is the concatenation $R_{n,k} := \widetilde{R_{n,k}} \ast H$.
\end{proof}

Since $\cP_{\mathsf{Disc_{n-1}}}$ is an $\mathbb{E}_{\mathsf{n-1}}$ operad, the following are straightforward verifications.

\begin{cor}
\mbox{}
\begin{itemize}
\item[(a)] For $n \ge 2$ and any symmetric monoidal category $\cD$ (enriched and tensored over $\mathsf{Top}$),  the map of operads $\varepsilon \colon \cP_{\mathsf{Disc_{n-1}}} \hookrightarrow \cP_{\mathsf{CD_n}}$ induces an equivalence of algebras
\[
\varepsilon^\ast \colon \mathsf{Alg}_{\cP_{\mathsf{CD_n}}} (\cD) \xrightarrow{ \; \simeq \;} \mathsf{Alg}_{\mathbb{E}_{\mathsf{n-1}}} (\cD).
\]
\item[(b)] Similarly, the forgetful map $\omega \colon \cP_{\mathsf{CD_n}} \to \cP_{\mathsf{Disc_n}}$ induces a map of algebras
\[
\omega^\ast \colon  \mathsf{Alg}_{\mathbb{E}_{\mathsf{n}}} (\cD) \to \mathsf{Alg}_{\cP_{\mathsf{CD_n}}} (\cD).
\]
\item[(c)] The composition  $\varepsilon^\ast \circ \omega^\ast \colon \mathsf{Alg}_{\mathbb{E}_{\mathsf{n}}} (\cD) \to \mathsf{Alg}_{\mathbb{E}_{\mathsf{n-1}}} (\cD)$ is equivalent to the map induced by the standard embedding $\mathsf{Disc_{n-1}} \hookrightarrow \mathsf{Disc_n}$.
\end{itemize}
\end{cor}

\subsection{Disjoint Causal Diamonds}\label{sect:cd}

In AQFT, more generally Lorentzian geometry, \emph{causal convexity} is often required.  The basic open of this type are the open (causal) diamonds of the form $I^+(p) \cap I^-(q)$, where $p,q$ are points and $I^\pm$ denotes the chronological future/past; similarly the closed diamonds, which are indeed the closure of the open diamonds, are of the form $J^+ (p) \cap J^-(q)$, the intersection of causal futures/pasts.  As we will see shortly, these (causal) diamonds can also be used to define an operad.  Moreover, the results of the previous section will apply these operads as well.

For clarity we will restrict to Minkowski space $\RR^{1,n-1}$, though the relevant notions make sense more generally for time oriented Lorentzian manifolds.  Recall, that a subset $U \subseteq \RR^{1,n-1}$ is \emph{causally convex} if every causal curve which begins and ends in $U$ is entirely contained in $U$.  Equivalently, $U$ contains all causal diamonds of its points.  Note that the (closed) Euclidean unit disc $D^2 \subset \RR^{1,1}$ is not causally convex as it does not contain the diamond $\Delta = J^-(\frac{\sqrt{2}}{2}, \frac{\sqrt{2}}{2}) \cap J^+ (\frac{\sqrt{2}}{2}, -\frac{\sqrt{2}}{2})$. (The same is true for the open disc, even using open diamonds, see Figure \ref{fig:3}.)

\begin{wrapfigure}{R}{0.35\textwidth}
\centering
\begin{tikzpicture}
   \draw[gray!60,fill=gray!20] (0,0) circle (1cm);
    \fill[pattern=crosshatch, pattern color = gray!75]plot coordinates{(0,-1.44)(-1.44,0)(0,1.44)(1.44,0)(0,-1.44)}; 
    \fill[gray!20] plot coordinates{(-1,0)(0,1)(1,0)(0,-1)(-1,0)};
    \draw node at (0,0) {$\Diamond$};
\end{tikzpicture}
\caption{The unit disc, its causal envelope, and the standard causal diamond $\Diamond$.}\label{fig:3}
\end{wrapfigure}

Let $\Diamond^n = I^-(1,0, \dotsc, 0) \cap I^+(-1,0, \dotsc, 0)$ be our standard causal/chronological diamond in $\RR^{1,n-1}$. (I will work with open diamonds, but one could similarly work with their closures if desired.)

\begin{definition}
Let $n \in \NN$.  Define \emph{the orthogonal category of causally disjoint diamonds} as follows.  The category $\mathsf{CDiam_n}$ has one object, $\mathrm{Hom} (\ast, \ast) = \mathrm{Rect}(\Diamond^n, \Diamond^n)$, and $(f,g) \in \perp_{\mathsf{CDiam_n}}$ if and only if $f(\Diamond^n)$ and $g(\Diamond^n)$ are causally disjoint (in $\RR^{1, n-1}$). Further, define \emph{the operad of little causal diamonds} by $\cP_{\mathsf{CDiam_n}} := \cP^\perp_{\mathsf{CDiam_n}}$.
\end{definition}

The deformation retractions of the previous section can be modified in a straightforward way to prove the following.

\begin{prop}
For $n \ge 2$, the operads $\cP_{\mathsf{CD_n}}$, $\cP_{\mathsf{CDiam_n}}$, and $\cP_{\mathsf{Disc_{n-1}}}$ are all homotopy equivalent.
\end{prop}

\section{Examples}

\begin{example}
As the operad $\cP_{\mathsf{CD_n}}$ is homotopy equivalent to an $\mathbb{E}_\mathsf{n-1}$-operad, we know that the associated homology operad $\cH ({\mathsf{CD_n}}):= H_\ast (\cP_{\mathsf{CD_n}}; \CC)$ is the Poisson-$(n-1)$ operad if $n \ge 3$ and $\cH (\mathsf{CD_2}) \cong \mathsf{As}$ is the associative operad. (This latter identification is actually quite simple if one considers carefully the causal structure on 2-dimensional Minkowski space.)

Recall the forgetful map $\omega \colon \mathsf{CD_n} \to \mathsf{Disc_n}$.  This map induces a map the same way in homology and a contravariant map at the level of algebras, i.e., $\omega^\ast \colon \mathsf{Alg}_{\cH ( \mathsf{Disc_n})} \to \mathsf{Alg}_{\cH (\mathsf{CD_n})}$. A natural class of Poisson-$n$ algebras arise from functions on shifted phase spaces, e.g., $\cO (T^\ast [n] M)$. If $n \ge 3$, the pullback of this algebra of functions on $\RR^{1,n-1}$ shifts the degree of the Poisson bracket by one.  In the case $n=2$, the pullback to causally disjoint discs in $\RR^{1,1}$ forgets the Poisson bracket entirely and only remembers the associative multiplication on functions.
\end{example}

\begin{example}
In \cite{GR}, Gwilliam and Rejzner construct a cosheaf of (dg) multilocal functionals $\mathbb{ML}oc$.  As they explain in {\it ibid} Section 5, $\mathbb{ML}oc$ when equipped with a certain differential generalizes the net of classical observables common to the AQFT literature.  Moreover, Gwilliam and Rejzner explain that there are structure maps
\[
\iota_{\{U_i\};V} \colon \bigotimes_i \mathbb{ML}oc_c (U_i) \to \mathbb{ML}oc_c (V)
\]
for disjoint opens $U_1, \dotsc , U_k$ contained in another open $V$.  Since these maps exist on compactly supported functionals, we can extend by zero to the boundary so we have structure maps defined for disjoint discs in a larger disc.  Now, it follows from classical results on the variational bicomplex, see Theorem 2.19 of \cite{GR}, that the inclusion of discs $ i \colon D \hookrightarrow D'$ induces a quasi-isomorphism $i_\ast \colon \mathbb{ML}oc (D) \xrightarrow{\; \sim \;} \mathbb{ML}oc (D')$. To summarize, in two dimensions, the cosheaf $\mathbb{ML}oc$ defines a $\cP_\mathsf{CD_2}$ algebra up to homotopy.  The corresponding $\mathbb{E}_\mathsf{1}$-structure on $\mathbb{ML}oc$ on $\RR^{1,1}$ is the familiar one which can be obtained by restricting to discs in the $\{t=0\}$ plane.
\end{example}

These examples suggest that the map $\omega \colon \mathsf{CD}_n \to \mathsf{Disc}_n$, and the associated map at the level of algebras, may provide a description of Wick rotation in the language of factorization algebras for field theories that are ``sufficiently topological."  More specifically, for a class of nice enough theories $\cE$ on $\RR^{n+1}$ the Wightman functions/observables defined on $\RR^{1,n}$ are the pullback via $\omega$ of the Schwinger functions/observables for $\cE$.  There remain many details to make this speculation precise and it remains work in progress.

\end{document}